\documentclass[a4paper,12pt,reqno]{amsart}

 \usepackage[T2A]{fontenc}
 \usepackage[utf8]{inputenc}
 \usepackage[ukrainian,english]{babel}
 \usepackage{amssymb,upref}
\usepackage{amsthm}

\usepackage{graphicx}
\usepackage{mathrsfs}

\usepackage[text={130mm,200mm}]{geometry}

\theoremstyle{plain}
\newtheorem{theorem}{Theorem}
\newtheorem*{theorem*}{Theorem}

\newtheorem*{corollary*}{Corollary}
\newtheorem{lemma}{Lemma}
\newtheorem*{lemma*}{Lemma}

\newtheorem*{proposition*}{Proposition}

\newtheorem*{conjecture*}{Conjecture}
\theoremstyle{definition}
\newtheorem{definition}{Definition}
\newtheorem*{definition*}{Definition}
\theoremstyle{remark}

\newtheorem*{remark*}{Remark}

\begin{document}

\title[Transformations and functions that preserve the asymptotic mean of digits]{Transformations and functions that preserve the asymptotic mean of digits in the ternary representation of a number}

\author{M. V. Pratsiovytyi}
\address[M. V. Pratsiovytyi]{Institute of Mathematics of NAS of Ukraine,  Dragomanov Ukrainian State University, Kyiv, Ukraine\\
ORCID 0000-0001-6130-9413}
\email{prats4444@gmail.com}
\author{S. O. Klymchuk}
\address[S. O. Klymchuk]{Institute of Mathematics of NAS of Ukrain, Kyiv, Ukraine\\
ORCID 0009-0005-3979-4543}
\email{svetaklymchuk@imath.kiev.ua}

\author{O. P. Makarchuk}
\address[O. P. Makarchuk]{Institute of Mathematics of NAS of Ukrain, Kyiv, Ukraine\\
ORCID 0000-0002-1001-8568}
\email{makarchuk@imath.kiev.ua}


\keywords{$s$--adic representation of a real number; digit frequency; asymptotic mean of digits; function that preserves the asymptotic mean of digits; transformation that preserves digit frequencies.}

\thanks{Scientific Journal of Drahomanov National Pedagogical University. Series 1. Physical and Mathematical Sciences. -- Kyiv:Drahomanov National Pedagogical University, 2013, No. 15, pp. 87--99.}

\begin{abstract}
In the paper we study transformations of the interval $[0;1)$ and functions that preserve the asymptotic mean $r$ of the digits in the $s$--adic representation of a number $x$,
$$r(x)=\lim\limits_{n\to\infty}\frac{1}{n}\sum\limits^{n}_{i=1}\alpha_i(x)$$
specify the necessary and sufficient conditions for a transformation to belong to this class

\end{abstract}

\maketitle
\section{Introduction}
When we consider the fractional part of a real number, that is, points of the interval $[0;1)$ we henceforth restrict ourselves to numbers $x\in[0;1)$.

Let $2\leqslant s$ denote a fixed natural number, $\mathcal{A}_s\equiv\{0,1,\ldots,s-1\}$ denote the alphabet of the $s$--adic representation and $L\equiv\mathcal{A}_s\times \mathcal{A}_s\times\ldots\times \mathcal{A}_s\ldots$ denote the space of sequences of s--adic digits. It is well known that for every real number $x\in[0;1]$ there exists a sequence
$(\alpha_k)$ such that $\left(\alpha_k\right)\in L$ and
$$
 x=\displaystyle\frac{\alpha_1}{s}+
   \displaystyle\frac{\alpha_2}{s^2}+\cdots+
   \displaystyle\frac{\alpha_k}{s^k}+\cdots\equiv\Delta^s_{\alpha_1\alpha_2\ldots\alpha_k\ldots}.
$$

We call the latter symbolic notation the \emph{$s$--adic representation} of the number $x$, and $\alpha_k=\alpha_k(x)$ its \emph{$k$th $s$--adic digit}. In general, the $k$th digit of a number $x$ is not a well-defined function of $x$, since the following equality holds
$$
\Delta^s_{c_1\ldots c_{k-1}c_k(0)}=\Delta^s_{c_1\ldots
c_{k-1}[c_k-1](s-1)},
$$
where $(i)$ denotes the period in the representation of the number. Numbers of this form are called \emph{$s$--adic rational}; they have exactly two $s$--adic representations and form a subset of the set of rational numbers. The remaining numbers have only one representation and are called \emph{$s$--adic irrational}. To ensure the correctness of the definition of the $k$th digit of a number, we agree to use only the first $s$--adic representation, namely, the one with period $(0)$. Then $\alpha_k(x)$ is a function correctly defined on $[0,1]$.

Recall that we define the \emph{frequency of the digit} $i \in \mathcal{A}s$ in the $s$--adic representation $\Delta^s{\alpha_1 \alpha_2 \ldots \alpha_k \ldots}$ of a number $x$ as the limit (if it exists):
$$
\nu_i(x)=\lim\limits_{n\to\infty}v_i^{(n)},
$$
where $v_i^{(n)}=n^{-1}N_i(x,n)$ s the relative frequency of the digit $i$, and
$N_i(x,n)=\# \{j:\, \alpha_j (x)=i, \, j\leqslant n\}$ is the quantity of digits $i$ in the $s$--adic representation of the number $x$ up to and including the $n$th position.

The concept of digit frequency in the $s$--adic representation of a number $x$, introduced by Émile Borel at the beginning of the 20th century \cite{Bor1}, proved to be highly productive. Researchers have used it effectively for various purposes, in particular: 1) to define different mathematical objects (sets, functions, measures, space transformations, etc.); 2) to solve a number of metric and probabilistic problems (establishing normal and abnormal properties of numbers, singularity of functions and probability distributions, etc.); 3) to solve problems in ergodic theory and fractal analysis.
In particular, one can use it to define self--similar but non--perfect \emph{Besicovitch--Eggleston sets}:
$$
E[\tau_0,\tau_1,\ldots,\tau_{s-1}]=\{x:x=\Delta^s_{\alpha_1\alpha_2\ldots\alpha_k\ldots},\,\nu_i(x)=\tau_i\geqslant0,\,\,i=\overline{0,s-1}\}.
$$
All Besicovitch--Eggleston sets have zero Lebesgue measure, except for the set
$$H_s=\{x:x=\Delta^s_{\alpha_1\alpha_2\ldots\alpha_k\ldots},\,\nu_i(x)=s^{-1},\,\,i=\overline{0,s-1}\},$$
of \emph{normal on the base $s$ numbers}, which, according to a well-known Borel's theorem \cite{Bor1}, has full Lebesgue measure.

The Hausdorff--Besicovitch fractal dimension of the set
$E[\tau_0,\tau_1,\ldots,\tau_{s-1}]$ is calculated \cite{Besic2,Egg,Bill} using the formula
$$
\alpha_0(E[\tau_0,\tau_1,\ldots,\tau_{s-1}])=
-\frac{\ln\tau_0^{\tau_0}\tau_1^{\tau_1}\ldots\tau_{s-1}^{\tau_{s-1}}}{\ln
s}.
$$

The \emph{asymptotic mean of digits} of the $s$--adic representation
$\Delta^s_{\alpha_1\alpha_2\ldots\alpha_k\ldots}$ of real number
$$x=\sum\limits^{\infty}_{k=1}\frac{\alpha_k}{s^k}\equiv\Delta^s_{\alpha_1\alpha_2\ldots\alpha_k\ldots},\:\:\alpha_k\in\mathcal{A}_s$$
is defined as the limit (if it exists):
$$r(x)=\lim\limits_{n\to\infty}\frac{1}{n}\sum\limits^{n}_{k=1}\alpha_k(x).$$

We introduced this concept in \cite{My1} and applied it to define and study the topological–metric properties of fractal sets \cite{My2}–\cite{My4}. It is closely connected to digit frequencies, because when the frequencies of all digits exist, the asymptotic mean of digits represents a certain averaged value. Specifically, if the $s$--adic representation of a number $x$ has digit frequencies $\nu_0, \nu_1, \ldots, \nu_{s-1}$, then its asymptotic mean of digits $r(x)$ is given by $$r(x)=\nu_1(x)+2\nu_2(x)+\ldots+(s-1)\nu_{s-1}(x).$$
Interestingly, this notion appeared in earlier works \cite{Besic2, Egg, HardiLitll, KennCoop} even without explicitly naming it.

In this paper, we introduce and study functions that are invariants of the asymptotic mean of digits in the $3$--adic ($s=3$) representation of a number, both everywhere and almost everywhere (in the sense of Lebesgue measure) on the interval $[0,1)$, as well as functions that preserve digit frequencies.

\section{Transformations of interval $[0,1)$ and functions that preserve digit frequencies}

Let $\Theta$ denote the set of numbers for which the frequencies  $\nu_0$,
$\nu_1$, $\ldots,$ $\nu_{s-1}$ of all digits exist.

Recall that a \emph{transformation} of a nonempty set $X$ is a bijective (one-to-one) mapping of the set onto itself.

\begin{definition}\label{o1}
We say that a \emph{transformation} $f$ of the half-interval $[0,1)$ \emph{preserves digit frequencies} if the equality
$$\nu_i\left(f(x)\right)=\nu_i(x),$$
holds for all $i\in \mathcal{A}_s$, $x\in\Theta$ and for $x\in [0,1)\setminus\Theta$ the frequencies $\nu_i(x)$ and $\nu_i(f(x))$ do not exist simultaneously.
\end{definition}

The identity transformation provides a trivial example of a function that preserves digit frequencies. The function $f$ defined on $[0,1)$ by the equality
$$
f(x)=\Delta^{s}_{\alpha_2(x)\alpha_1(x)\alpha_4(x)\alpha_3(x)\ldots}=y,
$$
since for any natural number $k$ we have
$$N_i(y,2k)=N_i(x,2k),$$
$$N_i(y,2k+1)=N_i(y,2k)+b, \text{ ~~~where~~~ } b\in\{0,1\}.$$
A simple example of a transformation that does not preserve digit frequencies is $$g(\Delta^3_{\alpha_1(x)\alpha_2(x)\ldots\alpha_n(x)\ldots})=\Delta^3_{\beta_1(x)\beta_2(x)\ldots\beta_n(x)\ldots},$$
where $\beta_i(x)=\alpha_i(x)+1(\mod 3)$ and moreover $g(\Delta^3_{(1)})\equiv \Delta^3_{(0)}$.
Indeed, for $x_0=\Delta^3_{(0)}$, $g(\Delta^3_{(0)})\equiv \Delta^3_{(1)}$ that is, $1=\nu_0(x_0)\neq\nu_0(g(x_0))=0$.

\begin{lemma}
The set $V$ of all transformations of the interval $[0;1)$ that preserve the frequencies of all digits in the $s$--adic representation of a number, with respect to the composition operation $\circ$ (superposition) of transformations, forms a noncommutative group.
\end{lemma}
\begin{proof}
It is well known that the family of all transformations of an arbitrary set with respect to the operation $\circ$,
 $$[f_2\circ f_1](x)=f_2(f_1(x))$$
forms a group. To prove the lemma, we apply the subgroup criterion.

Let $f_1,f_2\in V$ that is, $\nu_i(f_j(x))=\nu_i(x)$, $j\in\{1,2\}$ і $f_1(x)=x'$, $f_2(x)=x''$.
The closure of the set $V$ under the operation $\circ$ follows from the equalities 
$$\nu_i(f_2(f_1(x)))=\nu_i(f_2(x'))=\nu_i(x')=\nu_i(f_1(x))=\nu_i(x).$$

We show that for every transformation $f\in V$ the inverse transformation $f^{-1}$ also belongs to $V$. From $f\in V$ we have $\nu_i(f(x))=\nu_i(x)$. If $y=f(x)$ then $x=f^{-1}(y)$ and
$$\nu_i(y)=\nu_i(f(x))=\nu_i(x)=\nu_i(f^{-1}(y)).$$
Hence, $f^{-1}\in V$.

The group $(V,\circ)$ is noncommutative since for the transformations of $[0;1)$ defined by
$$f_1(x)=\Delta^{s}_{\alpha_2(x)\alpha_1(x)\alpha_4(x)\alpha_3(x)\ldots\alpha_{2k}\alpha_{2k-1}\ldots},$$ $$f_2(x)=\Delta^{s}_{\alpha_3(x)\alpha_2(x)\alpha_1(x)\alpha_4(x)\ldots\alpha_{2k
+1}\alpha_{2k}\alpha_{2k-1}\ldots}$$ we obtain
$[f_2\circ f_1](x)=\Delta^s_{\alpha_4(x)\alpha_1(x)\alpha_2(x)\alpha_3(x)\ldots}$ and
$[f_1\circ f_2](x)=\Delta^s_{\alpha_2(x)\alpha_1(x)\alpha_4(x)\alpha_1(x)\ldots}$. Therefore, $f_2\circ f_1\neq f_1\circ f_2$.
\end{proof}

\begin{definition}
If a function $f$, defined on the interval $[0;1)$ with values in $[0;1)$, satisfies the conditions of Definition \ref{o1}, then we say that the \emph{function $f$ preserves the digit frequencies}.
\end{definition}

Note that every transformation of $[0;1)$ is a function, but not every function defined on $[0;1)$ is a transformation.
Simple examples of functions that preserve digit frequencies and are not transformations are the following functions defined on $[0;1)$:
$$
\omega\left(\Delta^{s}_{\alpha_1\alpha_2\ldots\alpha_n\ldots}\right)\equiv
            \Delta^{s}_{\alpha_2\alpha_3\ldots\alpha_n\ldots},
$$
$$
\delta_i\left(\Delta^{s}_{\alpha_1\alpha_2\ldots\alpha_n\ldots}\right)\equiv
            \Delta^{s}_{i\alpha_1\alpha_2\ldots\alpha_n\ldots},
$$
where $i$ is a fixed element of the alphabet $\mathcal{A}_s$,
$$f(\Delta^{s}_{\alpha_1\alpha_2\ldots\alpha_j\ldots})=\Delta^{s}_{\alpha_1\ldots\alpha_{j-1}\alpha_{j+1}\alpha_{j}\alpha_{j+2}\ldots},$$
where $j$ is a predetermined natural number.

\section{Connection with functions that preserve fractal dimension}

From a group--theoretic perspective fractal geometry can be viewed as the theory of invariants of group transformations of the space $R^n$ that preserve fractal dimension; that is, transformations for which the fractal dimensions of the image and preimage of any Borel set are equal \cite{AlPrTor1, PrTorb1}.

\begin{theorem}
Some functions preserve digit frequencies and do not preserve the Hausdorff--Besicovitch fractal dimension.
\end{theorem}
\begin{proof} The set $\Theta$ is the union of all possible Besicovitch--Eggleston sets. We define the function $f$ for all $x \in [0,1)$ by the equality
$$f(x)=\begin{cases}\Delta^3_{\underbrace{0\ldots0}_{[\tau_0\cdot1]}
                \underbrace{1\ldots1}_{[\tau_1\cdot1]}
                \underbrace{2\ldots2}_{[\tau_2\cdot1]}\ldots
                \underbrace{0\ldots0}_{[\tau_0\cdot n]}
                \underbrace{1\ldots1}_{[\tau_1\cdot n]}
                \underbrace{2\ldots2}_{[\tau_2\cdot n]}\ldots}\equiv x',\text{ якщо }x\in\Theta,\\
                x,\text{ if }x\not\in\Theta,
       \end{cases}$$
where $\tau_0=\nu_0(x)$, $\tau_1=\nu_1(x)$, $\tau_2=\nu_2(x)$.

It is clear that $x'$ depends on the digit frequencies of the number $x$, that is, $x'=g(\nu_0(x),\nu_1(x),\nu_2(x))$ and the function $f$ is not a transformation of $[0,1)$, since $f(\Delta^3_{i\alpha_1\alpha_2\ldots\alpha_n\ldots})=f(\Delta^3_{\alpha_1\alpha_2\ldots\alpha_n\ldots})$.
Moreover, the image of the Besicovitch--Eggleston set $E[\tau_0, \tau_1, \tau_2]$ under the mapping $f$ is a single point $x'$.

We now show that $\nu_i(x')=\nu_i(x)$ if $x\in E[\tau_0,\tau_1,\tau_2]$.
For every sufficiently large natural number $n$, there exists $k_n\in N$ such that
$$\sum\limits^{k_n}_{j=1}j=\dfrac{k_n(k_n+1)}{2}\leqslant n<
\dfrac{(k_n+1)(k_n+2)}{2}=\sum\limits^{k_n+1}_{j=1}j.$$

Since $z-1<[z]\leqslant z$ for all $z\in R$, it follows that
$$\sum\limits^{n}_{j=1}([\tau_0\cdot j]+[\tau_1\cdot j]+[\tau_2\cdot j])\leqslant
\sum\limits^{n}_{j=1}(\tau_0+\tau_1+\tau_2)j=\dfrac{n(n+1)}{2},$$
$$\sum\limits^{n}_{j=1}([\tau_0\cdot j]+[\tau_1\cdot j]+[\tau_2\cdot j])>
\sum\limits^{n}_{j=1}((\tau_0+\tau_1+\tau_2)j-1)=\dfrac{n(n+1)}{2}-n.$$

For each $i \in {0,1,2}$, we have\\
$\dfrac{N_i(x',n)}{n}\geqslant\dfrac{\sum\limits^{k_n}_{j=1}[\tau_i\cdot
j]}{n}> \dfrac{\sum\limits^{k_n}_{j=1}(\tau_i\cdot
j-1)}{\dfrac{(k_n+1)(k_n+2)}{2}}=
\dfrac{\tau_i\dfrac{k_n(k_n+1)}{2}-k_n}{\dfrac{(k_n+1)(k_n+2)}{2}}\to
\tau_i (n\to\infty)$.\\
On the other hand, \\
$\dfrac{N_i(x',n)}{n}<\dfrac{\sum\limits^{k_n+1}_{j=1}[\tau_i\cdot
j]}{n}< \dfrac{\sum\limits^{k_n}_{j=1}\tau_i\cdot
j}{\dfrac{k_n(k_n+1)}{2}}=
\dfrac{\tau_i\dfrac{(k_n+1)(k_n+2)}{2}}{\dfrac{k_n(k_n+1)}{2}}\to
\tau_i (n\to\infty)$.

Hence, $\nu_i(x')=\tau_i$ for all
$i\in\{0,1,2\}$ which means that the function $f(x)$ preserves digit frequencies. However, if $\tau_0\tau_1\tau_2\neq0$ then taking into account the Besicovitch--Eggleston formula
 $\alpha_0(E[\tau_0,\tau_1,\tau_2])=-\dfrac{\ln \tau_0^{\tau_0}\tau_1^{\tau_1}\tau_2^{\tau_2}}{\ln 3}\neq0$.
 On the other hand,\\ $\alpha_0(\{f(x): x\in E[\tau_0,\tau_1,\tau_2]\})=\alpha_0(x')$=0, thus
$$\alpha_0(E[\tau_0,\tau_1,\tau_2])\neq\alpha_0(f(E[\tau_0,\tau_1,\tau_2])).$$
\end{proof}

It remains an open question whether there exist \emph{continuous} functions that preserve digit frequencies but do not preserve the Hausdorff--Besicovitch dimension.

\section{Transformations of $[0,1)$ and functions that preserve the asymptotic mean of digits but do not preserve digit frequencies}
\begin{definition}
We say that a function $f$ \emph{preserves the asymptotic mean of digits} if
$$r\left(f(x)\right)=r(x)$$
holds for all $x \in [0,1)$ for which $r(x)$ exists.
\end{definition}

\begin{definition}
We say that a function $f$ \emph{preserves the asymptotic mean of digits almost everywhere} (in the sense of Lebesgue measure) if the equality
$$r\left(f(x)\right)=r(x)$$ holds on a set of full Lebesgue measure.
\end{definition}

Since $r(x)=\nu_1(x)+\nu_2(x)$ then transformations of $[0,1)$ and functions that preserve digit frequencies also preserve the asymptotic mean of digits.

It is not difficult to give an example of a function that does not preserve the asymptotic mean of digits. In particular, such a function is defined by the equality
$$I(\Delta^s_{\alpha_1\alpha_2\ldots\alpha_n\ldots})\equiv \Delta^s_{[s-1-\alpha_1][s-1-\alpha_2]\ldots[s-1-\alpha_n]\ldots},$$
since
 $I(\Delta^3_{(0)})=\Delta^3_{(2)}$ і
$2=r\left(I(x)\right)\neq r(x)=0$.
The function $I$ is called the \emph{digit inverter} of the $s$--adic representation of a number.

\begin{lemma}\label{L1}
For the $3$--adic numeral system we have:
\begin{center}
if $r(x)=0$ then $\nu_0(x)=1$ and $\nu_1(x)=\nu_2(x)=0$;\\
if $r(x)=2$ then $\nu_2(x)=1$ and $\nu_0(x)=\nu_1(x)=0$.
\end{center}
\end{lemma}
\begin{proof}
Let $r(x)=0$, that is $\lim\limits_{n\to\infty} r_n(x)=0$. Since $$r_n(x)=v^{(n)}_1(x)+2v^{(n)}_2(x)\geqslant v^{(n)}_i(x)\geqslant0,$$ for both $i \in {1,2}$, it follows from $r(x) = 0$ that
$\nu_1(x)=\nu_2(x)=0$ and therefore $\nu_0(x)=1$.

Now let $r_n(x)=2$, that is, $\lim\limits_{n\to\infty} r_n(x)=2$. Since $$r_n(x)=v^{(n)}_1(x)+2v^{(n)}_2(x)=1-v^{(n)}_0(x)+v^{(n)}_2(x),$$ we have
$1\geqslant v^{(n)}_2(x)=r_n-1+v^{(n)}_0(x)\geqslant r_n-1$. Taking the limit as $\lim\limits_{n\to\infty} r_n(x)=2$, it follows that $\lim\limits_{n\to\infty} v^{(n)}_2(x)=1$. Consequently, for $i\in\{0,1\}$
$0\leqslant v^{(n)}_i(x)=1-v^{(n)}_2(x)-v^{(n)}_{1-i}(x)\leqslant 1-v^{(n)}_2(x)$. But $\lim\limits_{n\to\infty} 1-v^{(n)}_2(x)=0$, hence $\lim\limits_{n\to\infty} v^{(n)}_i(x)=0$, $i\in\{0,1\}$.
\end{proof}

\begin{theorem}
There exists no function $f:~[0;1)\to[0;1)$ for which both $r(x)=r(f(x))\,\,\,\,(2)$ and
$\sum\limits^{2}_{j=0}(\nu_j(x)-\nu_j(f(x)))^2\neq0\,\,\,\,(1)$ hold everywhere on the set $\Theta$.\\

In other words, any function $f:~[0;1)\to[0;1)$ that preserves the asymptotic mean of ternary digits must also preserve the frequencies of all digits; there is no number $x\in \Theta$ for which both conditions
$$\sum\limits^{2}_{j=0}(\nu_j(x)-\nu_j(f(x)))^2\neq0,\eqno(1)$$
$$r(x)=r(f(x)).\eqno(2)$$
are satisfied simultaneously.
\end{theorem}
\begin{proof}
Let $x\in E[0,0,1]$ then $r(f(x))=r(x)=2$. By Lemma \ref{L1} the number  $y=f(x)$ belongs to the set $E[0,0,1]$, which means that inequality (1) does not hold.

Similarly, if $x\in E[1,0,0]$, then we have the inclusion $f(x)\in
E[1,0,0]$.
\end{proof}

\section{Preservation of the asymptotic mean and digit frequencies almost everywhere}

A natural question arises about the existence of a function $f$ that preserves the asymptotic mean of digits of a number but does not preserve digit frequencies \emph{almost everywhere} (in the sense of Lebesgue measure) on the set $\Theta$ and hence also on the set $H$ of normal numbers.

We define the function $f$ at the point $x=\Delta^3_{\alpha_1\alpha_2\ldots\alpha_n\ldots}$ on $[0;1)$ by equality
$$f(x)=\begin{cases}
x, \text{~~~if~~~} x\in[0;1]\setminus H,\\
\Delta^3_{\beta_1\beta_2\ldots\beta_n\ldots}, \text{~~~if~~~} x\in H,
\end{cases}$$
where $\beta_n=\alpha_n$ whenever $\alpha_n\neq1$ and for the sequence $n_k$, such that $\alpha_{n_k}=1$ we obtain
$$
\beta_{n_k}=\begin{cases}
0\text{~~~if~~~} n_k=0(\mod 7),\\
2\text{~~~if~~~} n_k=1(\mod 7),\\
\alpha_{n_k} \text{~~~in the remaining cases.}
\end{cases}
$$
Hence the image of the function $f(x)$ at a normal point is obtained by replacing every seventh ternary digit ``1'' of the number with ``0'' and the following digit ``1'' with ``2''.

\begin{lemma}\label{L3}
The cardinality of the set $\{1,2,\ldots,n\}\cap\{7k: k\in N\}$ is
$\left[\dfrac{n}{7}\right]$, The cardinality of the set
$\{1,2,\ldots,n\}\cap\{7k+1: k\in N\}$ is
$\left[\dfrac{n-1}{7}\right]$.
\end{lemma}
\begin{proof}
Let $k\in N$ be such that $7k\leqslant n< 7(k+1)$. From this it follows that $k\leqslant
\dfrac{n}{7}< (k+1)$, hence $\left[\dfrac{n}{7}\right]=k$.
Let $k\in N$ be such that $7k+1\leqslant n< 7(k+1)+1$. From this it follows that
$k\leqslant \dfrac{n-1}{7}< (k+1)$, hence
$\left[\dfrac{n-1}{7}\right]=k$.
\end{proof}

\begin{theorem}\label{T2}
If $x\in H$, then $\nu_0(f(x))=\dfrac{8}{21}$,
$\nu_1(f(x))=\dfrac{8}{21}$ і $\nu_2(f(x))=\dfrac{5}{21}$.
\end{theorem}
\begin{proof}
Let  $k$ be a sufficiently large natural number and let $\varphi(k)\in N$ (where $\varphi$ is Euler’s function) be such that $$n_{\varphi(k)}\leqslant k\leqslant n_{\varphi(k)+1}.$$

Since $\nu_1(x)=\dfrac{1}{3}$ then
$\dfrac{1}{3}=\lim\limits_{n\to\infty}\dfrac{N(x,n_k)}{n_k}=\lim\limits_{k\to\infty}\dfrac{k}{n_k}$
and, taking into account Lemma \ref{L3}, we have \\
$\dfrac{N_1(f(x),k)}{k}=\dfrac{\varphi(k)-\left[\frac{\varphi(k)}{7}\right]-\left[\frac{\varphi(k)-1}{7}\right]}{k}\geqslant
\dfrac{\varphi(k)-\frac{\varphi(k)}{7}-\frac{\varphi(k)-1}{7}+2}{n_{\varphi(k)+1}}\to
\dfrac{1}{3}-\dfrac{1}{21}-\dfrac{1}{21}=\dfrac{5}{21}$ for
$k\to\infty$.\\
$\dfrac{N_1(f(x),k)}{k}=\dfrac{\varphi(k)-\left[\frac{\varphi(k)}{7}\right]-\left[\frac{\varphi(k)-1}{7}\right]}{k}\leqslant
\dfrac{\varphi(k)-\frac{\varphi(k)}{7}-\frac{\varphi(k)-1}{7}}{n_{\varphi(k)}}\to\dfrac{5}{21}$
for $k\to\infty$. \\
Hence, $\nu_1(f(x))=\dfrac{5}{21}$.

Let $N_0(x,k)$ denote the number of digits ``0'' in the block $\alpha_1\alpha_2\ldots\alpha_k$. Since $\nu_0=\dfrac{1}{3}$ we have
$\lim\limits_{k\to\infty}\dfrac{N_0(x,k)}{k}=\dfrac{1}{3}$ and taking into account Lemma \ref{L3}, it follows that

$\dfrac{N_0(f(x),k)}{k}=\dfrac{N_0(x,k)+\left[\frac{\varphi(k)}{7}\right]}{k}\geqslant
\dfrac{N_0(x,k)}{k}+\dfrac{\frac{\varphi(k)}{7}-1}{n_{\varphi(k)+1}}\to\dfrac{1}{3}+\dfrac{1}{21}=\dfrac{8}{21}$
for $k\to\infty$.

$\dfrac{N_0(f(x),k)}{k}=\dfrac{N_0(x,k)+\left[\frac{\varphi(k)}{7}\right]}{k}\leqslant
\dfrac{N_0(x,k)}{k}+\dfrac{\frac{\varphi(k)}{7}}{n_{\varphi(k)}}\to\dfrac{1}{3}+\dfrac{1}{21}=\dfrac{8}{21}$
for $k\to\infty$.\\ Hence, $\nu_0(f(x))=\dfrac{8}{21}$. Since
$\nu_0+\nu_1+\nu_2=1$ then
$\nu_2(f(x))=1-\dfrac{5}{21}-\dfrac{8}{21}=\dfrac{8}{21}$.
\end{proof}

If $x\in H$ then $r(x)=\dfrac{1}{3}+2\dfrac{1}{3}=1$ hence, by Theorem \ref{T2} we obtain $r(f(x))=\dfrac{5}{21}+2\dfrac{8}{21}=1$.

Thus, for any normal on the base 3 number $x$ the function $f$ preserves the asymptotic mean of digits almost everywhere. Function $f$ preserves the asymptotic mean of digits almost everywhere in the set $\Theta$ if and only if $$\lambda(\{f(x:x\in H)\}\cap
H)=1.$$

For each $x\in H$ we assign a corresponding $y\in H$ such that
$r(x)=r(y)=1$. Therefore, the set of functions $f: H\to H$ that preserve the asymptotic mean of digits almost everywhere has the cardinality of a hypercontinuum. It follows that the set of functions $f: [0;1]\to [0;1]$
preserving the asymptotic mean of digits almost everywhere also has hypercontinuum cardinality.

\section{Preservation of the asymptotic digit mean by a function when digit frequencies do not exist}

Let us construct an example of a function $f(x)$ that preserves the asymptotic mean of ternary digits, i.e. $r(x)=r(f(x))$ for any
$x\in H$ and for which the digit frequencies $\nu_i(f(x))$ do not exist for all $i\in\{0,1,2\}$.

Consider the sequence $(\beta_n)$ defined by $\beta_n= \Delta^3_{\underbrace{0\ldots0}_{1!}\underbrace{1\ldots1}_{2!}\ldots\underbrace{0\ldots0}_{(2n-1)!}\underbrace{1\ldots1}_{2n!}\ldots}$.

Let
$f(x)=\begin{cases}
            x, \text{ if } x\in[0;1)\setminus  H,\\
            \Delta^3_{\underbrace{0\ldots0}_{[\tau_{01}1^x]}
                      \underbrace{1\ldots1}_{[\tau_{11}1^x]}
                      \underbrace{2\ldots2}_{[\tau_{21}1^x]}\ldots
                      \underbrace{0\ldots0}_{[\tau_{0n}n^x]}
                      \underbrace{1\ldots1}_{[\tau_{1n}n^x]}
                      \underbrace{2\ldots2}_{[\tau_{2n}n^x]}\ldots}, \text{ if } x\in H,
      \end{cases}$
where
$$
(\tau_{0n};\tau_{1n};\tau_{2n})=
          \begin{cases}
            (0,4;0,2;0,4), \text{ if } \beta_n=0,\\
            (0,3;0,4;0,3), \text{ if } \beta_n=1.
          \end{cases}
$$

\begin{lemma}
For any $\alpha>0$ the following equalities hold:
$$\lim\limits_{n\to\infty}\dfrac{(n+1)^{1+\alpha}}{1^{1+\alpha}+2^{1+\alpha}+\ldots+n^{1+\alpha}}=0,\:\:\:\:
\lim\limits_{n\to\infty}\dfrac{n^{2+\alpha}}{1^{1+\alpha}+2^{1+\alpha}+\ldots+n^{1+\alpha}}=2+\alpha.$$
\end{lemma}
\begin{proof}
Let $s_n=\sum\limits^{n}_{i=1}i^{\alpha+1}$ then, by Lagrange’s theorem we have \\
$$(n+1)^{2+\alpha}-n^{2+\alpha}=z_n^{1+\alpha}(2+\alpha),\text{  where  }z_n\in[n;n+1], \text{ hence }$$
$$\dfrac{(n+1)^{2+\alpha}-n^{2+\alpha}}{n^{1+\alpha}}=(2+\alpha)\left(\dfrac{z_n}{n}\right)^{1+\alpha}\to2+\alpha.$$
Since $\dfrac{(n+1)^{1+\alpha}-n^{1+\alpha}}{s_{n}-s_{n-1}}=\left(1+\dfrac{1}{n}\right)^{1+\alpha}-1\to0$ for $n\to\infty$then by Stolz’s theorem we have $\lim\limits_{n\to\infty}\dfrac{(n+1)^{1+\alpha}}{s_n}=0$.
\end{proof}

\begin{lemma}
The asymptotic mean $r(f(x))$  of digits $f(x)$ equals $1$ for all $x\in H$.
\end{lemma}
\begin{proof}
For any $n\in N$ there exists a natural number $k_n\in N$ such that
$$\sum\limits^{k_n}_{j=1}\sum\limits^{2}_{i=0}[\tau_{ij}j^{x+1}]\leqslant n< \sum\limits^{k_n+1}_{j=1}\sum\limits^{2}_{i=0}[\tau_{ij}j^{x+1}].$$
Then as $n\to \infty$ we have\\
$\begin{array}{ll}
r_n(f(x))&>\dfrac{\sum\limits^{k_n}_{j=1}([\tau_{1j}j^{x+1}]+2[\tau_{2j}j^{x+1}])}{\sum\limits^{k_n+1}_{j=1}\sum\limits^{2}_{i=0}[\tau_{ij}j^{x+1}]}>
\dfrac{\sum\limits^{k_n}_{j=1}(\tau_{1j}+2\tau_{2j})j^{x+1}-k_n}{\sum\limits^{k_n+1}_{j=1}(\tau_{0j}+\tau_{1j}+\tau_{2j})j^{x+1}}=\dfrac{\sum\limits^{k_n}_{j=1}j^{x+1}-k_n}{\sum\limits^{k_n+1}_{j=1}j^{x+1}}=\\
&=1-\dfrac{(k_n+1)^{x+1}}{\sum\limits^{k_n+1}_{j=1}j^{x+1}}-\dfrac{k_n}{\sum\limits^{k_n+1}_{j=1}j^{x+1}}\to1,
\end{array}$ \\
since
$\dfrac{k_n}{\sum\limits^{k_n+1}_{j=1}j^{x+1}}<\dfrac{k_n}{(k_n+1)^{x+1}}\to
0$ and $\dfrac{(k_n+1)^{x+1}}{\sum\limits^{k_n+1}_{j=1}j^{x+1}}=\left(1+\dfrac{1}{k_n}\right)^{x+1}\cdot\dfrac{k_n^{x+1}}{\sum\limits^{k_n+1}_{j=1}j^{x+1}}
\to0$.\\
$\begin{array}{ll}
r_n(f(x))&<\dfrac{\sum\limits^{k_n+1}_{j=1}([\tau_{1j}j^{x+1}]+2[\tau_{2j}j^{x+1}])}{\sum\limits^{k_n}_{j=1}\sum\limits^{2}_{i=0}[\tau_{ij}j^{x+1}]}<
\dfrac{\sum\limits^{k_n+1}_{j=1}(\tau_{1j}+2\tau_{2j})j^{x+1}}{\sum\limits^{k_n}_{j=1}(\tau_{0j}+\tau_{1j}+\tau_{2j})j^{x+1}-k_n}=\\
&=\dfrac{\sum\limits^{k_n+1}_{j=1}j^{x+1}}{\sum\limits^{k_n}_{j=1}j^{x+1}-k_n}=
\dfrac{1+\frac{(k_n+1)^{x+1}}{\sum\limits^{k_n}_{j=1}j^{x+1}}}{1-\frac{k_m}{\sum\limits^{k_n}_{j=1}}}\to1.
\end{array}$\\
Hence, $\lim\limits_{n\to\infty}r_n(f(x))=1$.
\end{proof}

\begin{lemma}
For any number $x\in H$ the corresponding value $f(x)$ does not have well-defined digit frequencies for $i\in\{0,1,2\}$ of all $3$--adic digits.
\end{lemma}
\begin{proof}
Since as $n\to\infty$ we have\\
$$\dfrac{\sum\limits^{n+1}_{k=1}(2k-1)!-\sum\limits^{n}_{k=1}(2k-1)!}{\sum\limits^{2n+1}_{k=1}k!-\sum\limits^{2n-1}_{k=1}k!}=
\dfrac{(2n+1)!}{(2n+1)!+(2n)!}=\dfrac{1}{1+\dfrac{1}{2n+1}}\to 1,$$
then by Stolz’s theorem we obtain
$\dfrac{1!+3!+\ldots+(2n-1)!}{1!+2!+\ldots+(2n-1)!}\to 1$ hence
$\dfrac{2!+\ldots+(2n-2)!}{1!+\ldots+(2n-1)!}\to 0$.

Taking into account that
$\dfrac{\sum\limits^{n+1}_{k=1}(2k-1)!-\sum\limits^{n}_{k=1}(2k-1)!}{\sum\limits^{2n+2}_{k=1}k!-\sum\limits^{2n}_{k=1}k!}=
\dfrac{(2n+1)!}{(2n+2)!+(2n+1)!}=\dfrac{1}{2n+3}\to 0$ as
$n\to\infty$ then
$\dfrac{1!+3!+\ldots+(2n-1)!}{1!+2!+\ldots+(2n)!}\to 0$ hence
$\dfrac{2!+\ldots+(2n-2)!}{1!+\ldots+(2n)!}\to 0$.

Let \\
$$\begin{array}{ll}
&k_n=1!+2!+\ldots+(2n-1)!,\\
&l_n=\sum\limits^{k_n}_{j=1}[\tau_{0j}j^x]+[\tau_{1j}j^x]+[\tau_{2j}j^x],
\end{array}$$ then as $n\to\infty$ we have\\
$\dfrac{N_0(f(x),l_n)}{l_n}=\dfrac{\sum\limits^{k_n}_{j=1}[\tau_{0j}j^{x+1}]}{l_n}>
\dfrac{\sum\limits^{k_n}_{j=1}\tau_{0j}j^{x+1}-3k_n}{\sum\limits^{k_n}_{j=1}j^{x+1}}$. (It is clear that
$\dfrac{k_n}{\sum\limits^{k_n}_{j=1}j^{x+1}}\to 0$).\\
$$\sum\limits^{k_n}_{j=1}\tau_{0j}j^{x+1}=0,4\sum\limits^{h_n}_{j=1}b_{j}^{x+1}+0,3\sum\limits^{s_n}_{j=1}l_{j}^{x+1},$$
where\\
$$\begin{array}{ll}
&\sum\limits^{h_n}_{j=1}b_{j}^{x+1}+\sum\limits^{s_n}_{j=1}l_{j}^{x+1}=\sum\limits^{k_n}_{j=1}j^{x+1},\\
&h_n=1!+3!+\ldots+(2n-1)!,\\
&s_n=2!+4!+\ldots+(2n-2)!.
\end{array}$$
 We then have
$\dfrac{\sum\limits^{s_n}_{j=1}l_{j}^{x+1}}{\sum\limits^{k_n}_{j=1}j^{x+1}}<
 \dfrac{s_n\cdot k_n^{x+1}}{\sum\limits^{k_n}_{j=1}j^{x+1}}=\dfrac{s_n}{k_n}\cdot\dfrac{k_n^{x+2}}{\sum\limits^{k_n}_{j=1}j^{x+1}}\to 0\cdot(2+x)=0$. \\
 Therefore $\lim\limits_{n\to\infty}\dfrac{\sum\limits^{s_n}_{j=1}l_{j}^{x+1}}{\sum\limits^{k_n}_{j=1}j^{x+1}}=0$,
 $\lim\limits_{n\to\infty}\dfrac{\sum\limits^{h_n}_{j=1}b_{j}^{x+1}}{\sum\limits^{k_n}_{j=1}j^{x+1}}=1$,
 $\lim\limits_{n\to\infty}\dfrac{\sum\limits^{k_n}_{j=1}\tau_{0j}j^{x+1}-3k_n}{\sum\limits^{k_n}_{j=1}j^{x+1}}=0,4$.

On the other hand \\
$\dfrac{N_0(f(x),l_n)}{l_n}<\dfrac{\sum\limits^{k_n}_{j=1}\tau_{0j}j^{x+1}}{\sum\limits^{k_n}_{j=1}j^{x+1}+k_n}=
\dfrac{\frac{\sum\limits^{k_n}_{j=1}\tau_{0j}j^{x+1}}{\sum\limits^{k_n}_{j=1}j^{x+1}}}{1+\frac{k_n}{\sum\limits^{k_n}_{j=1}j^{x+1}}}
\to 0,4$ при $n\to \infty$.

Hence, $\lim\limits_{n\to\infty}\dfrac{N_0(f(x),l_n)}{l_n}=0,4$.

Now let
$$\begin{array}{ll}
&k^*_n=1!+2!+\ldots+(2n)!,\\
&l^*_n=\sum\limits^{k^*_n}_{j=1}[\tau_{0j}j^x]+[\tau_{1j}j^x]+[\tau_{2j}j^x],\\
&h^*_n=1!+3!+\ldots+(2n-1)!,\\
&s^*_n=2!+4!+\ldots+(2n)!.
\end{array}$$
We obtain,
$\sum\limits^{k^*_n}_{j=1}\tau_{0j}j^{x+1}=0,4\sum\limits^{h^*_n}_{j=1}b_{j}^{*x+1}+0,3\sum\limits^{s^*_n}_{j=1}l_{j}^{*x+1}$,
where
$\sum\limits^{h^*_n}_{j=1}b_{j}^{*x+1}+\sum\limits^{s^*_n}_{j=1}l_{j}^{*x+1}=\sum\limits^{k^*_n}_{j=1}j^{x+1}$.\\
Thus
$\dfrac{\sum\limits^{h^*_n}_{j=1}b_{j}^{*x+1}}{\sum\limits^{k^*_n}_{j=1}j^{x+1}}<\dfrac{h^*_n}{k^*_n}\cdot\dfrac{k^{*x+2}_n}{\sum\limits^{k^*_n}_{j=1}j^{x+1}}
\to 0\cdot(x+2)=0$ при $n\to\infty$.

Hence, $\lim\limits_{n\to\infty}\dfrac{\sum\limits^{h^*_n}_{j=1}b_{j}^{*x+1}}{\sum\limits^{k^*_n}_{j=1}j^{x+1}}=0$,
$\lim\limits_{n\to\infty}\dfrac{\sum\limits^{s^*_n}_{j=1}l_{j}^{*x+1}}{\sum\limits^{k^*_n}_{j=1}j^{x+1}}=1$.\\
$\dfrac{N_0(f(x),l^*_n)}{l^*_n}>\dfrac{\sum\limits^{k^*_n}_{j=1}\tau_{0j}j^{x+1}-k^*_n}{\sum\limits^{k^*_n}_{j=1}j^{x+1}}\to 0,4\cdot0+0,3\cdot1=0,3$ при $n\to\infty$,\\
$\dfrac{N_0(f(x),l^*_n)}{l^*_n}<\dfrac{\sum\limits^{k^*_n}_{j=1}\tau_{0j}j^{x+1}}{\sum\limits^{k^*_n}_{j=1}j^{x+1}+k^*_n}\to
0,4\cdot0+0,3\cdot1=0,3$ при $n\to\infty$.

Therefore, $\lim\limits_{n\to\infty}\dfrac{N_0(f(x),l^*_n)}{l^*_n}=0,3$ and the frequency $\nu_0(f(x))$ does not exist.

Similarly we can show that
$\lim\limits_{n\to\infty}\dfrac{N_j(f(x),l_n)}{l_n}=\tau_{j1}$ and
$\lim\limits_{n\to\infty}\dfrac{N_j(f(x),l^*_n)}{l^*_n}=\tau_{j2}$,
 and therefore also $\nu_j(f(x))$ does not exist for $j\in\{1,2\}$.
\end{proof}

\end{document}